\documentclass[12pt, a4paper]{article}

\usepackage{setspace}
\usepackage{latexsym}
\usepackage{amssymb}
\usepackage{amsthm}
\usepackage{amsfonts}
\usepackage{amsmath}
\usepackage[margin=1in,nohead]{geometry}
\usepackage{enumerate}
\usepackage{color}
\usepackage{hyperref}

\usepackage{tikz-cd}

\newtheorem{thm}{Theorem}[section]
\newtheorem{cor}[thm]{Corollary}
\newtheorem{lem}[thm]{Lemma}

\newtheorem{con}[thm]{Construction}
\theoremstyle{definition}
\newtheorem{ex}[thm]{Example}

\usetikzlibrary{decorations.markings}
\usetikzlibrary{arrows}

\tikzstyle{vertex}=[circle,draw,inner sep=0pt, minimum size=6pt]

\DeclareMathOperator\Aut{{\rm Aut}}

\DeclareMathOperator\Cos{{\rm Cos}}
\DeclareMathOperator\Sym{{\rm Sym}}
\DeclareMathOperator\GF{{\rm GF}}
\newcommand\Wr{{\rm \,wr\, }}
\newcommand{\ep}{\varepsilon}

\newcommand{\norml}{\vartriangleleft}
\newcommand{\liso}{\lesssim}

\DeclareMathOperator\PSL{{\rm PSL}}
\DeclareMathOperator\AGL{{\rm AGL}}
\DeclareMathOperator\PGL{{\rm PGL}}
\DeclareMathOperator\GL{{\rm GL}}
\DeclareMathOperator\PGammaL{{\rm P}\Gamma{\rm L}}
\DeclareMathOperator\PA{{\rm PA}}
\DeclareMathOperator\M{{\rm M}}
\DeclareMathOperator\Out{{\rm Out}}

\DeclareMathOperator{\soc}{soc}

\newcommand\GAP{\textsf{GAP}}

\title{Locally s-arc-transitive graphs arising\\ from product action}
\author{ Michael Giudici \\ Department of Mathematics and Statistics\\ The University of Western Australia\\ Perth WA 6009, Australia\\ \\
Eric Swartz \\Department of Mathematics\\
 William \& Mary\\ P.O. Box 8795\\
  Williamsburg, VA 23187-8795, USA}

\begin{document}

\maketitle

\begin{abstract}
 We study locally $s$-arc-transitive graphs arising from the quasiprimitive product action (PA).  We prove that, for any locally $(G,2)$-arc-transitive graph with $G$ acting quasiprimitively with type PA on both $G$-orbits of vertices, the group $G$ does not act primitively on either orbit.  Moreover, we construct the first examples of locally $s$-arc-transitive graphs of PA type that are not standard double covers of $s$-arc-transitive graphs of PA type, answering the existence question for these graphs.
\end{abstract}

\section{Introduction}

For an integer $s\geqslant 1$, an $s$-arc in a graph $\Gamma$ is an $(s+1)$-tuple $(\alpha_0,\alpha_1,\ldots,\alpha_s)$ of vertices such that $\alpha_i\sim \alpha_{i+1}$ and $\alpha_i\neq \alpha_{i+2}$ for each $i$. We say that $\Gamma$ is \emph{$s$-arc-transitive} if the automorphism group of $\Gamma$ acts transitively on the set of all $s$-arcs. If $\Gamma$ is $s$-arc-transitive and each $(s-1)$-arc can be extended to an $s$-arc then any $s$-arc-transitive graph is also $(s-1)$-arc-transitive.
 The study of $s$-arc-transitive graphs goes back to the pioneering work of Tutte \cite{tutte,tutte2}, who showed that if $\Gamma$ has valency three then $s\leqslant 5$.  Weiss \cite{weiss1} later showed that if the valency restriction is relaxed to allow valency at least three then $s\leqslant 7$, with equality holding for the generalised hexagons arising from the groups $G_2(q)$ for $q=3^f$.

Praeger \cite{quasiprim1} initiated a programme for the study of finite connected $s$-arc-transitive graphs by first showing that if $G\leqslant \Aut(\Gamma)$ acts transitively on the set of all $s$-arcs of $\Gamma$ and $N\norml G$ has at least three orbits on the set of vertices, then the quotient graph $\Gamma_N$ whose vertices are the orbits of $N$ is also $s$-arc-transitive.  Moreover, $\Gamma$ is a cover of $\Gamma_N$. This reduces the study of finite connected $(G,s)$-arc-transitive graphs to two basic types:
\begin{itemize}
\item those where $G$ is \emph{quasiprimitive} on the set of vertices, that is, where all nontrivial normal subgroups of $G$ are transitive on vertices;
\item those where $G$ is \emph{biquasiprimitive} on the set of vertices, that is, where all nontrivial normal subgroups of $G$ have at most two orbits on vertices and there is a normal subgroup with two orbits.
\end{itemize}
Praeger showed that of the eight types of  finite quasiprimitive groups, only four --- HA (affine),  TW (twisted wreath), AS (almost simple) and PA (product action) --- can act 2-arc-transitively on a graph \cite{quasiprim1}. We use the types of quasiprimitive groups as given in \cite{quasiprim2} and define type PA, the main focus of this paper, in Section \ref{sec:PA}. These are slight variations on the types of primitive permutation groups given by the O'Nan--Scott Theorem.  All graphs of type HA  were classified by Praeger and Ivanov \cite{affine2trans} while those of type TW were studied by Baddeley \cite{2arctw}. The 2-arc-transitive graphs for some families of almost simple groups have all been classified, for example the Suzuki groups \cite{2arcsuz}, Ree groups \cite{2arcree} and $\PSL(2,q)$ \cite{2arcpsl}. The first examples of 2-arc-transitive graphs of PA type were given by Li and Seress \cite{prodaction} and studied further by Li, Seress, and Song \cite{Li2015}.  Another family of quasiprimitive 2-arc-transitive graphs of PA type were constructed by Li, Ling, and Wu in \cite{LiLingWu}.

In the biquasiprimitive case the graph is bipartite and such graphs were investigated in \cite{bipartite2arc,biquasi}. An alternative way to study such graphs is via the notion of local $s$-arc-transitivity. We say that a graph $\Gamma$ is \emph{locally $(G,s)$-arc-transitive} for a group $G\leqslant \Aut(\Gamma)$ if for each vertex $\alpha$, the vertex stabiliser $G_\alpha$ acts transitively on the set of all $s$-arcs starting at $\alpha$.  If $G$ also acts transitively on the set of vertices then  $\Gamma$ is $s$-arc-transitive. If $\Gamma$ is locally $(G,s)$-arc-transitive but $G$ is intransitive on the set of vertices, then $G$ has two orbits on vertices and $\Gamma$ is bipartite.  One way to construct locally $s$-arc-transitive graphs is to start with an $s$-arc-transitive graph  $\Gamma$ and take its standard double cover $\Sigma$, which has vertex set $V\Gamma\times\{1,2\}$ and $(\alpha,i)\sim(\beta,j)$ precisely when $i\neq j$ and $\alpha\sim \beta$ in $\Gamma$.   Then $\Aut(\Gamma)$ acts as automorphisms on $\Sigma$ with two orbits on vertices and $\Sigma$ is  locally $(\Aut(\Gamma),s)$-arc-transitive \cite{localsarc}.

If $\Gamma$ is a bipartite graph and $G\leqslant \Aut(\Gamma)$ acts transitively on the set of vertices, then $\Gamma$ is locally $(G^+,s)$-arc-transitive where $G^+$ is  the index two subgroup that stabilises each part of the bipartition. Hence the study of locally $s$-arc-transitive graphs encompasses the study of all bipartite $s$-arc-transitive graphs and hence  the biquasiprimitive case in Praeger's programme. It is also a wider class of graphs as the known generalised octagons are locally 9-arc-transitive but not vertex-transitive, and it has been shown by van Bon and Stellmacher \cite{sleq9} that this is best possible. 

A programme for the study of finite connected locally $s$-arc-transitive graphs was mapped out by Giudici, Li and Praeger \cite{localsarc}. If $\Gamma$ is locally $(G,s)$-arc-transitive with $G$ having two orbits on vertices and $N\norml G$ is intransitive on both $G$-orbits, then the quotient graph $\Gamma_N$ is also locally  $s$-arc-transitive.  Moreover, $\Gamma$ is a cover of $\Gamma_N$. This reduces the study of finite connected locally $(G,s)$-arc-transitive graphs for which $G$ is vertex-intransitive into two basic types:
\begin{itemize}
\item those where $G$ is quasiprimitive on each of its two orbits on vertices;
\item those where $G$ is quasiprimitive on only one of its two orbits on vertices.
\end{itemize}
In the second case, it was shown \cite{localsarc} that the quasiprimitive action must be of type HA, HS, AS, PA or TW. These were further studied in \cite{localstar} where all examples where the quasiprimitive action has type HS or PA were classified. An infinite family of examples where the quasiprimitive action has type TW was given by Kaja and Morgan \cite{KM}. In the first case,  either the two quasiprimitive actions have the same quasiprimitive type and are one of HA, AS, TW or PA, or they are different with one of type SD and one of type PA \cite{localsarc}. All 2-arc-transitive graphs of the latter type were classified in \cite{localdifferent} and there are locally 5-arc-transitive examples in this case \cite{5arc}. It was shown in \cite[Lemma 3.2]{biprimcubic} that all  locally 2-arc-transitive graphs where the quasiprimitive action is of type HA on both orbits are actually vertex-transitive and so classified in \cite{affine2trans}. All locally $(G,2)$-arc-transitive graphs have been classified in the cases where $G$ is an almost simple group whose socle is a Ree group \cite{localrees}, Suzuki group \cite{localsuz}, or $\PSL(2,q)$ \cite{BDSL}, while the sporadic group case was studied in \cite{localsporadic}.  Examples also exist in the PA and TW cases as we can take standard double covers of $s$-arc-transitive graphs of type PA and TW respectively.

The aim of this paper is to study locally $s$-arc-transitive graphs of PA type.  We prove that, for any locally $(G,2)$-arc-transitive graph with $G$ acting quasiprimitively with type PA on both $G$-orbits of vertices, the group $G$ does not act primitively on either orbit. Moreover, in the spirit of \cite{prodaction}, we solve the existence problem for locally $2$-arc-transitive graphs of PA type.  In particular, we construct the first examples of locally $s$-arc-transitive graphs of PA type that are not standard double covers of $s$-arc-transitive graphs of PA type.

\section{PA type}\label{sec:PA}

Let $G$ act quasiprimitively on a set $\Omega$. We say that $G$ has type \emph{PA} if there exists a $G$-invariant partition $\mathcal{B}$ of $\Omega$ such that $G$ acts faithfully on $\mathcal{B}$ and we can identify $\mathcal{B}$ with $\Delta^k$ for some set $\Delta$ and $k\geqslant 2$ such that $G\leqslant H\Wr S_k$ acts in the usual product action of a wreath product on $\Delta^k$,  where $H\leqslant \Sym(\Delta)$ is an almost simple group acting quasiprimitively on $\Delta$.  Moreover, if $T=\soc(H)$ then $G$ has a unique minimal normal subgroup $N=T^k$. Note that since $G$ is quasiprimitive,  $N$ acts transitively on $\Omega$ and hence on $\mathcal{B}$. Thus $G=NG_\alpha=NG_B$, where $B \in \mathcal{B}$ is a block containing $\alpha \in \Omega$. As $N$ is minimal normal in $G$ we have that $G$ transitively permutes the simple direct factors of $N$ and hence so do both $G_\alpha$ and $G_B$. Thus given $B=(\delta,\ldots,\delta)\in\mathcal{B}$ we may assume that $N_B=T_\delta^k$ and for $\alpha\in B$ we have that $N_\alpha$ is a subdirect product of $N_B$, that is, the projection of $N_\alpha$ onto each direct factor is isomorphic to $T_\delta$.  

Let $R=T_\delta$. Following the terminology of \cite{prodaction}, if $N_\alpha \cong R$ then we call $N_\alpha$ a \emph{diagonal subgroup} of $N_B=R^k$. Then there exists automorphisms $\varphi_2,\varphi_3,\ldots,\varphi_k$ of $R$ such that 
$$N_\alpha=\{(t,t^{\varphi_2},\ldots,t^{\varphi_k})\mid t\in R\}$$
If each of the $\varphi_i$ is the trivial automorphism then we call $N_\alpha$ a \emph{straight diagonal subgroup} while if some $\varphi_i$ is nontrivial then we call $N_\alpha$ a \emph{twisted diagonal subgroup}. Furthermore, if $N_\alpha\not\cong R$ then we refer to $N_\alpha$ as being a \emph{nondiagonal subgroup.} We refer to the quasiprimitive permutation group $G$ of type PA as being of \emph{straight diagonal}, \emph{twisted diagonal} or nondiagonal type according to the type of $N_\alpha$.

Note that unlike for primitive groups of type PA, $G$ does not necessarily preserve a product structure on $\Omega$, only on some $G$-invariant partition $\mathcal{B}$. Indeed the following result shows that for locally 2-arc-transitive graphs this partition must be nontrivial on each of the bipartite halves.

\begin{thm}\label{thm:PA1}
Let $\Gamma$ be a locally $(G,2)$-arc-transitive connected graph with $G$ quasiprimitive of type PA on both orbits $\Omega_1$ and $\Omega_2$. Let $N=T^k=\soc(G)$ and for $i=1,2$, let $\mathcal{B}_i$ be a $G$-invariant partition of $\Omega_i$ such that $G$ preserves a product structure $\Delta_i^k$ on each $\mathcal{B}_i$.   Then $\mathcal{B}_i\neq \Omega_i$ for each $i$.
\end{thm}
\begin{proof}
Suppose that $\mathcal{B}_i=\Delta_i$ for some $i$. Without loss of generality suppose that $i=1$. Also note that there is an almost simple group $H$ with socle $T$ such that $G\leqslant H\Wr S_k$. 

Let $\alpha=(\omega,\ldots,\omega)\in\Omega_1$. Then $N_{\alpha}=T_\omega^k$ with $T_\omega\neq 1$ and $G_{\alpha}=G\cap (H_\omega\Wr S_k)$. 
 By \cite[Lemma 3.2]{localsarc}, $G_{\alpha}^{\Gamma(\alpha)}$ is 2-transitive so either all neighbours of $\alpha$ lie in same block of $\mathcal{B}_2$ or in distinct blocks.  If they all lie in the same block then for each $\beta\in\Delta_1$ we have that the neighbours of $\beta$ lie in the same block. However, this contradicts $\Gamma$ being connected. Hence for each $\alpha\in\Delta_1$, the neighbours of $\alpha$ lie in distinct blocks. Hence $G_\alpha$ acts 2-transitively on the set $X$ of blocks of $\mathcal{B}_2$ that contain neighbours of $\alpha$. By \cite[Lemma 6.2]{localsarc}, $N_\alpha^{\Gamma(\alpha)}$ is a transitive subgroup of the 2-transitive group $G_\alpha^{\Gamma(\alpha)}$ and so $N_\alpha$ also acts transitively on $X$.  Let $B=(\delta_1,\delta_2,\ldots,\delta_k)\in\mathcal{B}_2$ be a block containing a neighbour~$\gamma$ of $\alpha$. Then $X=(\delta_1,\delta_2,\ldots,\delta_k)^{N_\alpha}=\delta_1^{T_{\omega}}\times\delta_2^{T_{\omega}}\times\cdots\times \delta_k^{T_{\omega}}$. By \cite[Theorem 1.1(b)]{PS}, the stabiliser $G_1$ in $G$ of the first simple direct factor of $N$ projects onto $H$ in the first coordinate and so $(G_1)_\alpha$ projects onto $H_\omega$ in the first coordinate. Hence $\delta_1^{T_\omega}=\delta_1^{H_{\omega}}$. Since $G_\alpha\leqslant H_\omega\Wr S_k$ and transitively permutes the $k$ simple direct factors of $N$, it follows that $\delta_i^{T_\omega}=\delta_1^{T_\omega}$ for each $i$. In particular, $X=A^k$ for some set $A$ and we could have chosen $B=(\delta,\ldots,\delta)$ for some $\delta\in\Delta_2$. Thus $G_{\alpha\gamma}\leqslant G_{\alpha, B}\leqslant H_{\omega\delta}\Wr S_k$. However, for $\delta'\in A\backslash\{\delta\}$ there is no element of $H_{\omega\delta}\Wr S_k$ mapping $(\delta',\delta,\ldots,\delta)$ to $(\delta',\delta',\delta,\ldots,\delta)$, contradicting $G_\alpha$ acting 2-transitively on $X$. Thus $\mathcal{B}_1\neq \Omega_1$. 
\end{proof}
\begin{cor}
Let $\Gamma$ be a locally $(G,2)$-arc-transitive connected graph with $G$ quasiprimitive of type PA on both orbits. Then $G$ is not primitive on either orbit.
\end{cor}

\section{Constructions}

Let $G$ be a finite group with subgroups $L$ and $R$. Let $\Delta_1$ be the set $[G:L]$ of right cosets of $L$ in $G$ and $\Delta_2$ be the set $[G:R]$ of left cosets of $R$ in $G$. We define the \emph{coset graph} $\Gamma=\Cos(G,L,R)$ to be the bipartite graph with vertex set the disjoint union $\Delta_1\cup\Delta_2$ such that $\{Lx,Ry\}$ is an edge if and only if $Lx\cap Ry\neq \varnothing$, or equivalently $xy^{-1}\in LR$. Then $G$ acts by right multiplication on both $\Delta_1$ and $\Delta_2$, and induces automorphisms of $\Gamma$. Note that the vertices in $\Delta_1$ have valency $|L:L\cap R|$ while the vertices in $\Delta_2$ have valency $|R:L\cap R|$. We say that $\Gamma$ has  \emph{valency} $\{|L:L\cap R|,|R:L\cap R|\}$.  Conversely, if $\Gamma$ is a graph and $G\leqslant\Aut(\Gamma)$ acts transitively on the set of edges of $\Gamma$ but not on the set of vertices then $\Gamma$  can be constructed in this way \cite[Lemma 3.7]{localsarc}. We refer to the triple $(L,R,L\cap R)$ as the associated \emph{amalgam}.

 We collect the following properties of coset graphs. We say that a subgroup $H$ of a group $G$ is \emph{core-free} if $\cap _{g\in G} H^g=1$.

\begin{lem}\label{lem:cosets} \cite[Lemma 3.7]{localsarc} Let $G$ be a group with proper subgroups $L$ and $R$, and let $\Gamma=\Cos(G,L,R)$. 
\begin{enumerate}
\item $\Gamma$ is connected if and only if $G=\langle L,R\rangle$.
\item $G$ acts faithfully on both $L$ and $R$ if and only if both $L$ and $R$ are core free in $G$.
\item $G$ acts transitively on the set of edges of $\Gamma$.
\item $\Gamma$ is locally $(G,2)$-arc-transitive if and only if $L$ acts 2-transitively on $[L:L\cap R]$ and $R$ acts 2-transitively on $[R:L\cap R]$.
\end{enumerate}
\end{lem}

We also need the following result, which essentially follows from the definition of a \textit{completion} (and the \textit{universal completion}) of an amalgam (see \cite{trivalent}) and results on covers of graphs (see, e.g., \cite[Chapter 19]{Biggs}).  The result is truly ``folklore'': while it seems to be taken for granted in the field, we also cannot find an explicit proof in the literature.  We have included a proof here provided by Luke Morgan \cite{LukeLemma}.

\begin{lem}\label{lem:Luke}
If $\Gamma$ is a locally $s$-arc-transitive graph with amalgam $(L,R,L\cap R)$ and $s\geqslant 2$, then any other graph with amalgam $(L,R,L\cap R)$ is locally $s$-arc-transitive.
\end{lem}
\begin{proof}
 Let $G:= L *_{L\cap R} R$ be the universal completion of $(L,R,L\cap R)$ and let $\Gamma^*$ denote the universal tree on which $G$ acts edge-transitively.  We identify $L$ and $R$ with their images in $G$, and label an edge $\{\alpha, \beta\}$ so that $G_\alpha = L$, $G_\beta =R$, and $G_{\alpha \beta} =L\cap R$.  Since $\Gamma$ is locally $s$-arc-transitive for $s\geqslant 2$, it is locally 2-arc-transitive and so the actions of $L$ on the set of right cosets of $L\cap R$ in $L$, and of $R$ on the set of right cosets of $L\cap R$ in $R$ are 2-transitive \cite[Lemma 3.2]{localsarc}. In particular, $\Gamma$ is locally $(G,2)$-arc-transitive. 
 
 Now let $\Sigma$ be a graph with edge-transitive group of automorphisms $H$ such that the amalgam $(H_\gamma, H_\delta, H_{\gamma \delta})$ is isomorphic to $(L,R,L\cap R)$, where $\{\gamma, \delta\}$ is an edge of $\Sigma$.  By the universal property of $G$ and of $\Gamma^*$, there is a map $\phi: G \to H$ such that the appropriate diagram commutes.  Let $N$ be the kernel of $\phi$.  Then, $\Sigma = \Gamma^*_N$, the quotient graph, and the kernel of the action of $G$ on $\Sigma$ is exactly $N$.
 
 In particular, $\phi(G_\alpha) = H_\gamma$ and $\phi(G_\beta) = H_\delta$.  Further, since $\phi(G_{\alpha \beta}) = H_{\gamma \delta}$, we have commutative diagrams of the following groups:
 
 \begin{center}

	\begin{tikzcd}
G_\alpha \arrow[r] \arrow[d]
& G_\alpha^{\Gamma^*(\alpha)} \arrow[d] 
& G_\beta \arrow[r] \arrow[d]
& G_\beta^{\Gamma^*(\beta)} \arrow[d]\\
H_\gamma \arrow[r]
&  H_\gamma^{\Sigma(\gamma)}
& H_\delta \arrow[r]
& H_\delta^{\Sigma(\delta)}
\end{tikzcd},\\
\end{center}
where $G_\alpha^{\Gamma^*(\alpha)}$ denotes the induced action of $G_\alpha$ on $\Gamma^*(\alpha)$, etc.

We now claim that for $\ep = \gamma, \delta$ and $\zeta \in \Gamma^*(\ep)$, we have $\zeta^N \cap \Gamma^*(\ep) = \{\zeta\}$.  Indeed, this follows since $|G_\alpha:G_{\alpha \beta}| = |H_\gamma:H_{\gamma \delta}|$ and $|G_\beta:G_{\alpha \beta}| = |H_\delta:H_{\gamma \delta}|$.

Now suppose $\Gamma^*$ is locally $(G,r)$-arc-transitive and $\Sigma$ is locally $(H,t)$-arc-transitive.  By \cite[Lemma 5.1(3)]{localsarc}, we have $t \geqslant r$.  

Assume that $r< t$.  We will show that $\Gamma^*$ would be locally $(G, r+1)$-arc-transitive in this case, contradicting the maximality of $r$.

Suppose $P$ and $P'$ are $(r+1)$-paths in $\Gamma^*$ with initial vertex $\alpha$ or $\beta$.  Since $s \geqslant1$, without loss of generality we may assume $P = (\alpha, \beta_1, \dots, \beta_r, \beta_{r+1})$ and $P' = (\alpha, \beta_1, \dots, \beta_r, \beta_{r+1}')$, where $\beta_1 = \beta$.

Consider the images of $P^N$ and $(P')^N$ in $\Sigma$.  Note that the images are two $(s+1)$-paths, since the equality $\beta_{i-1}^N = \beta_{i+i}^N$ would contradict our claim above.  Hence, there is $h \in H_\gamma$ such that $(P^N)^h = (P')^N$.  Since $\phi(G_\alpha) = H_\gamma$, we can take $h = \phi(g)$ for $g \in G_\alpha$, so $g$ fixes $\alpha$.  Now, $(P^N)^h = (P')^N$ implies $(\beta^N)^g = \beta^N$.  Thus, $g$ fixes $\beta^N$, and, since $g$ fixes $\alpha$, $g$ fixes the unique vertex in $\Gamma^*(\alpha) \cap \beta^N$, which is $\beta$; so, $g \in G_{\alpha \beta}$.  Continuing in this way, we see that $g \in G_{\alpha \beta_1 \dots \beta_r}$.  Now, $(\beta_{r+1}^N)^h = (\beta_{r+1}')^N$, and so $\beta_{r+1}^g$ lies in the $N$-orbit of $\beta_{r+1}'$, and at the same time must be adjacent to $\beta_r$, since $g \in G_{\beta_r}$.  Once more, the claim implies $\beta_{r+1}^g = \beta_{r+1}'$.

We have thus shown that $G_\alpha$ is transitive on $(s+1)$-arcs with initial vertex $u$.  A similar argument establishes that same result for $G_\beta$, and hence $\Gamma^*$ is locally $(G, r+1)$-arc-transitive.  This contradicts the maximality of $r$, and, therefore, $r= t$, as desired. In particular, taking $\Sigma=\Gamma$ we see that $r=s$. Hence $\Gamma^*$, and so any graph with amalgam $(L,R,L\cap R)$, is locally $s$-arc-transitive.
\end{proof}

Lemma \ref{lem:cosets} enables us to construct locally $(G,2)$-arc-transitive graphs where $G$ has two orbits $\Delta_1$ and $\Delta_2$ on vertices and acts quasiprimitively of type PA on each.  Recall the three types straight diagonal, twisted diagonal and nondiagonal of quasiprimitive groups of type PA. Analogously to \cite{prodaction}, we refer to a locally $(G,2)$-arc-transitive graph $\Gamma$ where $G$ is quasiprimitive of type PA on each orbit by the type of the two PA actions. For example, if $G$ is of straight diagonal type on $\Delta_1$ and twisted diagonal type on $\Delta_2$ then we refer to $\Gamma$ as being of \emph{straight-twisted type}.

\subsection{Straight-twisted type}

\begin{con}\label{con:strtwi}
We begin with the following: let $(L,R, L \cap R)$ be an amalgam for a locally $s$-arc-transitive graph, and suppose further that $L=L_1\rtimes K$ and $R=R_1\rtimes K$ such that  $K$ acts trivially on $R_1$.

Let $H$ be an almost simple group with socle $T$, and subgroups $H_1$ and $H_2$ such that 
\begin{itemize}
 \item $H_1 \cong L_1$, $H_2 \cong R_1$, $H_1 \cap H_2 \cong L_1 \cap R_1$,
 \item $H = \langle H_1, H_2 \rangle$, and 
 \item not all automorphisms of $L_1$ in $K$ extend to automorphisms of $T$.
\end{itemize}

We will abuse notation slightly and assume $L_1,R_1 \leqslant H.$  Let $k=|K|$ and let $F=\{f:K\rightarrow H\}\cong H^k$.
For each $\ell\in L_1$ and $r\in R_1$, define $f_{\ell},f_r\in F$ such that $f_{\ell}(\kappa)=\ell^\kappa$ and $f_r(\kappa)=r$ for all $\kappa\in K$. Furthermore, we let $N_{\alpha}:=\{f_{\ell}\mid \ell \in L_1\}\cong L_1$ and $N_{\beta}=\{f_r\mid r\in R_1\}\cong R_1$. Since $K$ acts trivially on $R_1$, we have that $N_{\alpha}\cap N_{\beta}=\{f_r\mid r\in R_1\cap L_1\}\cong L_1\cap R_1.$ Let $N:=\langle N_\alpha,N_{\beta}\rangle$.

Now $K$ acts on $F$ via $f^{\sigma}(\kappa)=f(\sigma\kappa)$ for each $\sigma,\kappa\in K$. Then for $\ell\in L_1$ we have that $(f_{\ell})^\sigma (\kappa)= f_\ell(\sigma\kappa)=\ell^{\sigma\kappa}=f_{\ell^\sigma}(\kappa)$. Hence $(f_{\ell})^\sigma=f_{\ell^\sigma}$ and so $K$ normalises $N_{\alpha}$. Similarly, $(f_r)^{\sigma}=f_r$ for all $r\in R_1$ so $K$ normalises $N_{\beta}$ and hence also $N$. Define $G_{\alpha}:= N_\alpha\rtimes K$, $G_{\beta}:= N_\beta\rtimes K$ and $G:= \langle G_{\alpha}, G_{\beta} \rangle.$ Let $\Gamma=\Cos(G,G_\alpha,G_\beta)$.
\end{con}

\begin{lem}
\label{lem:straighttwisted}
Let $\Gamma$ be a graph yielded by  Construction \ref{con:strtwi}. Then $\Gamma$ is  a connected locally $(G,s)$-arc-transitive graph  such that $G$ acts quasiprimitively with type $\PA$ on each orbit of vertices.  Moreover, the action of $G$ on $[G:G_{\beta}]$ is straight diagonal, and the action of $G$ on $[G:G_{\alpha}]$ is twisted diagonal, that is, $\Gamma$ is of straight-twisted type.  
\end{lem}

\begin{proof}
Let $F_T=\{f\in F\mid f(\kappa)\in T \textrm{ for all } \kappa\in K\}\cong T^k$. For each $\kappa\in K$, let 
$$\begin{array}{llll}
   \pi_{\kappa}: &F &\rightarrow & H \\
                 &f &\mapsto &f(\kappa)   
  \end{array}$$
Since $\langle R_1,L_1\rangle=H$, we have that $\pi_{\kappa}(N)=H$ for all $\kappa\in K$ and so by \cite[p. 328, Lemma]{scott},
$N\cap F_T$ is a direct product of diagonal subgroups, each isomorphic to $T$. Since not every $\kappa$ extends to an automorphism of $T$ it follows that $N\cap F_T$ is not itself a diagonal subgroup and so $N \cap F_T\cong T^j$ for some integer $2 \leqslant j \leqslant k$. 

Since the action of $K$ on $N_{\alpha}$ is isomorphic to the action of $K$ on $L_1$ we see that $G_{\alpha}\cong L$ and similarly, $G_{\beta}\cong R$.  Moreover, $G_{\alpha} \cap G_{\beta} \cong \langle L_1 \cap R_1, K \rangle = L \cap R.$  Therefore $\Gamma:= \Cos(G, G_{\alpha}, G_{\beta})$ is a connected graph with amalgam $(L,R, L \cap R)$ and is thus a locally $s$-arc-transitive graph.

Finally, since $K$ transitively permutes the simple direct factors of $F_T$ it also transitively permutes the simple direct factors of $N\cap F_T$. Thus $\soc(G) \cong T^j$ and $G \liso H \Wr S_j$ for some integer $j \geqslant 2.$  Since $\pi_{\kappa}(N_{\alpha})=L_1$ for all $\kappa\in K$ it follows that $N_{\alpha}$ is a subdirect subgroup of $L_1^j$ and similarly, $N_{\beta}$ is a subdirect subgroup of $R_1^j.$  Therefore, $G$ acts quasiprimitively with type $\PA$ on both $[G:G_{\alpha}]$ and $[G:G_{\beta}]$, and, by construction, the action of $G$ on $[G:G_{\beta}]$ is straight diagonal, and the action of $G$ on $[G:G_{\alpha}]$ is twisted diagonal.
\end{proof}

\begin{ex}
This example is based on \cite[Example 4.1]{prodaction}.  First, $(\AGL(1,5) \times C_2, S_3 \times C_4, C_4 \times C_2)$ is an amalgam admitting a locally 2-arc-transitive connected graph of valency $\{3,5\}$: indeed, a \textsc{Gap} computation shows that  in the group $S_7$ we can take $L=\langle (4,5,6,7), (3,4,5,7,6), (1,2) \rangle \cong \AGL(1,5) \times C_2$ and $R= \langle (1,2), (2,5), (1,2,5) \rangle \cong S_3 \times C_4$ such that  $\langle L,R \rangle = S_7$, and $L\cap R \cong C_4 \times C_2$ \cite{GAP4}.

Let $T = \PSL(2,p)$, where $p$ is a prime and $p \equiv \pm 1 \pmod {60}$.  Thus we may select $H < T$ such that $H \cong D_{60}$, with $H = \langle h,d \mid h^{30} = d^2 = 1, h^d = h^{-1} \rangle.$  First, define $L_1:= \langle h^3 \rangle \cong C_{10} \cong C_5 \times C_2.$  Noting that $H$ has a subgroup $B := \langle h^{15}, d \rangle \cong C_2^2$, there exists an element $x$ of $T$ such that $B^x = B$ and $d^x = h^{15}$ \cite{dickson}.  Define $R_1:= \langle (h^{10})^x, d^x \rangle$ to be a subgroup of $H^x$ isomorphic to $S_3$.  Hence $\langle L_1, R_1 \rangle = T$ and $L_1 \cap R_1 = C_2$  Finally, the order four elements of $\AGL(1,5)$ cannot be extended to automorphisms of $T$ since $\Aut(T) = \PGL(2,p)$ has no elements of order four normalising but not centralising a subgroup of order five.  Thus we let $K = \langle k \rangle \cong C_4$ and $L=L_1\rtimes K$. Note, as in \cite[Example 4.1]{prodaction}, that the action of $k^2$ on elements of $T$ is the same as conjugation by $d$.  Therefore, by Lemma \ref{lem:straighttwisted}, there is a locally $2$-arc-transitive graph with amalgam $(\AGL(1,5) \times C_2, S_3 \times C_4, C_4 \times C_2)$ of straight-twisted type.  
\end{ex}

\begin{thm}
There is an infinite family of locally 5-arc-transitive graphs with valencies $\{4,5\}$ of straight-twisted type.
\end{thm} 

\begin{proof}
By \cite{localsporadic}, there is an amalgam admitting a locally 5-arc-transitive connected graph  of valency $\{4,5\}$ from the Mathieu group $\M_{24}$, with $L = C_2^4\rtimes(A_4 \times C_3),$ $R = A_5 \times A_4,$  and $L \cap R = A_4 \times A_4.$  Note that $L=L_1\rtimes K$ and $R=R_1\times K$ where $L_1=C_2^4\rtimes C_3$, $R_1=A_5$ and $K=A_4$.

Let $n\geqslant 2$ be an integer and $T=\PSL(2,2^{2n})$. Then $T$ contains a subgroup $R_1 \cong A_5 \cong \PSL(2,4)$ (see \cite{dickson}, for instance).  Furthermore, $T$ contains a subgroup $Y$ isomorphic to $C_2^{2n}\rtimes C_{2^{2n}-1}$, and $2^{2n}-1 \equiv 0 \pmod 3$.  Let $Y = Y_2\rtimes Y_1$, where $Y_2 \cong C_2^{2n}$ and $Y_1 = \langle y_1 \rangle \cong C_{2^{2n}-1}$.  Thus $Y_1$ has a cyclic group of order three, which we will denote by $Y_3 = \langle y_1^{(2^{2n}-1)/3} \rangle$, acting semiregularly on the nonidentity elements of $Y_2$.  Moreover, we may choose $R_1$ such that $Y_0 := R_1 \cap Y \cong A_4$ and $Y_3 \leqslant Y_0$.  By \cite[Theorem 260]{dickson}, we see that $N_T(Y_0) \leqslant Y$, and, noting that $Y_1$ acts regularly on the nonidentity elements of $Y_2$, we see that $N_T(Y_0) = Y_0$.  By \cite[Theorem 255]{dickson}, for each divisor $m$ of $2n$, all subfield subgroups of $T$ isomorphic to $\PSL(2,2^m)$ are conjugate. This implies that $Y_0$ is contained in a unique subfield subgroup $T_m$ isomorphic to $\PSL(2,2^{m})$ for each divisor $m$ of $2n$, $m$ even (if $m$ is odd, then $2^2-1 = 3$ does not divide $2^m-1$).  Note also that this implies that the maximal subgroup of $T_m$ isomorphic to $C_2^m\rtimes C_{2^m-1}$ is actually $T_m \cap Y$. We claim that no subfield subgroup $T_m$ containing $Y_0$, for $m$ a proper even divisor of $2n$, also contains $Y_0^{y_1}$.  If some $T_m$ contains $Y_0^{y_1}$, since the elements of order two in $Y_0$ and $Y_0^{y_1}$ commute and $Y_0 \cap Y_0^{y_1} = Y_3$ we have that $\langle Y_0, Y_0^{y_1}\rangle \leqslant T_m \cap Y  \cong C_2^m\rtimes C_{2^{m}-1}$ where $T_m \cap Y_1$ acts regularly on the nonidentity elements of $T_m \cap Y_2.$  However, $Y_1$ acts regularly on the nonidentity elements of $Y_2$, so $y_1$ is the unique element of $Y_1$ mapping, say, $y_2 \in Y_0 \cap Y_2$ to $y_2^{y_1} \in Y_0^{y_1} \cap Y_2$. On the other hand, $y_1 \not\in T_m \cap Y_1 = \langle y_1^{(2^{2n}-1)/(2^m - 1)}\rangle,$ so we have a contradiction. 

Let $L_1 := \langle Y_0, Y_0'^{y_1}\rangle$. Then $L_1 \cong 2^4{:}3$ (SmallGroup(48,50) in the \GAP \cite{GAP4} small groups library) which is isomorphic to the subgroup $L_1$ in $L$, hence the abuse of notation. Moreover, $L_1\cap R_1\cong A_4$ and, since $L_1$ is not contained in any subfield subgroup, we have that $T = \langle L_1, R_1 \rangle$.  
Since $\PGammaL(2,2^{2n})$ does not contain a subgroup isomorphic to $L$ (\cite[Theorem 260]{dickson} and noting that the outer automorphism group of $\PSL(2,2^n)$ is cyclic), it follows that not all automorphisms of $L_1$ in $L$ extend to automorphisms of $T$. Hence by Lemma \ref{lem:straighttwisted}, Construction \ref{con:strtwi} yields a locally 5-arc-transitive graph of straight-twisted type.
\end{proof}

\subsection{Twisted-twisted type}\label{sec:twtw}

If $G$ acts quasiprimitively of type straight PA type on a set $\Omega$,  then there exists $\alpha\in\Omega$ such that $N_\alpha=\{(r,r,\ldots,r)\mid r\in R\}$, where $N=T^k$ is the unique minimal normal subgroup of $G$. If $g=(t_1,t_2,\ldots,t_k)\in R^k\leqslant N$ then $N_{\alpha^g}=(N_\alpha)^g=\{(r^{t_1},r^{t_2},\ldots,r^{t_k})\mid r\in R\}$, which is a twisted diagonal subgroup if $t_i\notin C_T(R)$ for some $i$.  Thus the examples given in the previous section can also be viewed as being of twisted-twisted type.
However, if $G$ acts quasiprimitively of type twisted PA on a set $\Omega$ then $N_\alpha$ is a twisted diagonal subgroup of $R^k$ for some $R$ but there may not be a $\beta\in\Omega$ such that $N_\beta$ is a straight diagonal subgroup. Thus not all twisted-twisted type examples arise in this way.  In this section we give an alternative construction.

\begin{con}\label{con:twitwi}
Let $(L,R, L \cap R)$ be an amalgam for a locally $s$-arc-transitive graph, and suppose further that $L =L_1\rtimes K$ and $R =R_1\rtimes K$ such that $K=K_L\times K_R$ where $K_L\leqslant \Out(L_1)$, $K_L$ acts trivially on $R_1$, $K_R\leqslant \Out(R_1)$ and $K_R$ acts trivially on $L_1$.  Let $H$ be an almost simple group with socle $T$, and subgroups $H_1$ and $H_2$ such that 
\begin{itemize}
 \item $H_1 \cong L_1$, $H_2 \cong R_1$, $H_1 \cap H_2 \cong L_1 \cap R_1$,
 \item $H = \langle H_1, H_2 \rangle$, and 
 \item not all elements of $K$ extend to automorphisms of $T$.
\end{itemize}

We will abuse notation slightly and assume $L_1,R_1 \leqslant H.$  Let $k=|K|$ and let $F=\{f:K\rightarrow H\}\cong H^k$.
For each $\ell\in L_1\cup R_1$, define $f_{\ell}\in F$ such that $f_{\ell}(\kappa)=\ell^\kappa$  for all $\kappa\in K$. Furthermore, we let $N_{\alpha}:=\{f_{\ell}\mid \ell \in L_1\}\cong L_1$ and $N_{\beta}=\{f_r\mid r\in R_1\}\cong R_1$. Moreover, $N_{\alpha}\cap N_{\beta}=\{f_r\mid r\in R_1\cap L_1\}\cong L_1\cap R_1.$ Let $N:=\langle N_\alpha,N_{\beta}\rangle$.

Now $K$ acts on $F$ via $f^{\sigma}(\kappa)=f(\sigma\kappa)$ for each $\sigma,\kappa\in K$. As in Construction \ref{con:strtwi}, $K$ normalises both $N_{\alpha}$ and $N_{\beta}$, and hence also $N$.
 Define $G_{\alpha}:= N_{\alpha},\rtimes K$, $G_{\beta}:=  N_{\beta}\rtimes  K $ and $G:= \langle G_{\alpha}, G_{\beta} \rangle.$ Let $\Gamma=\Cos(G,G_\alpha,G_\beta)$.
\end{con}

\begin{lem}
\label{lem:twistedtwisted}
Let $\Gamma$ be a graph yielded by  Construction \ref{con:twitwi}. Then $\Gamma$ is  a connected locally $(G,s)$-arc-transitive graph  such that $G$ acts quasiprimitively with type $\PA$ on each orbit on vertices.  Moreover, the action of $G$ on both $[G:G_{\alpha}]$ and $[G:G_{\beta}]$ is twisted diagonal, that is, $\Gamma$ is of twisted-twisted type.  
\end{lem}

\begin{proof}
Proof is analogous to that of Lemma \ref{lem:straighttwisted}.
\end{proof}

\begin{ex}
First, $(C_{71}{:}C_{70} \times C_9, C_{19}{:}C_{18} \times C_{35}, C_{630})$ is an amalgam that admits a locally 2-arc-transitive graph; indeed, if $G = A_{89}$,
\begin{align*}
 L := \langle &(1,2,8,28,14,30,34,3,20,54,36,33,40,41,9,56,26,51,60,18,42,29,39,17,46,58,\\
 &47,10,15,70,62,13,32,59,57,31,66,22,24,67,48,27,35,50,45,12,23,11,52,4,64,\\
    &7,53,25,16,61,21,44,6,5,68,71,19,55,38,69,65,49,63,43,37),\\ &(2,3,4,\dots,71)(72,73,\dots,89)\rangle,
\end{align*}
and
\begin{align*}
 R:= \langle &(1,72,73,85,74,88,86,78,75,80,89,84,87,77,79,83,76,82,81),\\
 &(2,3,4,\dots,71)(72,73,\dots,89)\rangle,
 \end{align*}
then, using $\GAP$, we see that $L \cong C_{71}{:}C_{70} \times C_9$, $R \cong C_{19}{:}C_{18} \times C_{35}$, $L \cap R \cong C_{630}$, $\langle L, R \rangle = G$, and by Lemma \ref{lem:cosets}, the coset graph $\Cos(G,L,R)$ is a connected locally $(G,2)$-arc-transitive graph.

Let $T = \mathbb{M},$ the Monster Group.  By \cite{primederange}, $T$ contains subgroups $L_1 \cong D_{142}$ and $R_1 \cong D_{38}$, and $L_1$ and $R_1$ may be selected such that $L_1 \cap R_1 \cong C_2$ (here, the element of order two is of type 2B). By \cite{Wilson} we see that $\mathbb{M}$ does not have a maximal subgroup of order divisible by 71 and 19. Thus $\langle L_1,R_1\rangle=T$. Let $K =C_{315} = C_{35} \times C_9,$ and since $T$ does not contain an element of order $315$ \cite{atlas}, not all elements of  $K$ lift to an automorphism of $T$. Therefore, by Lemma \ref{lem:twistedtwisted}, Construction \ref{con:twitwi} yields a locally 2-arc-transitive graph $\Gamma$ with amalgam $(C_{71}{:}C_{70} \times C_9, C_{19}{:}C_{18} \times C_{35}, C_{630})$  of twisted-twisted type with valencies $\{71, 19 \}.$
\end{ex}

\subsection{Straight-nondiagonal type}

We first include an example of an \textit{equidistant linear code} from \cite{prodaction}, which proves useful in later constructions.  A \textit{linear (n,k)-code} $C$ over $\GF(q)$ is a $k$-dimensional subspace of $\GF(q)^n$, a codeword has \textit{weight} $w$ if it has exactly $w$ nonzero coordinates, and a code $C$ is \textit{equidistant} if all nonzero codewords have the same weight.  

\begin{ex}\cite[Example 5.1]{prodaction}
\label{ex:code}
Let $V = \GF(3)^4$, and let
\[C = \langle (1,1,1,0), (1,2,0,1)\rangle < V. \]
Then, $C$ is a linear $(4,2)$-code, and it contains eight nonzero code words:
\[(1,1,1,0), (1,2,0,1), (2,0,1,1), (0,2,1,2), (2,2,2,0), (2,1,0,2), (1,0,2,2), (0,1,2,1), \]
and hence $C$ is equidistant of weight $3$.

Let $\tau = (\sigma, 1, \sigma, \sigma)(1,2,3,4) \in \GL(1,3) \Wr S_4 < \GL(V)$.  Then, $\tau^4 = (\sigma, \sigma, \sigma, \sigma)$, $|\tau| = 8$, and $\tau$ permutes the eight nonzero words of $C$ in the order given above.
\end{ex}

Our next result constructs examples of straight-nondiagonal type.

\begin{thm}\label{thm:strnon}
For each integer $n\geqslant 3$, there exists a locally 2-arc-transitive graph of straight-nondiagonal type with  valencies $\{n,9\}$.
\end{thm}

\begin{proof}
We adapt the construction of \cite[Lemma 5.2]{prodaction}.  Let $H = S_{n+2}$.  Then $H$ contains subgroups $L \cong S_2 \times S_{n}$ and $R \cong S_3 \times S_{n-1}$ such that $\langle L,R \rangle = H$ and $L \cap R \cong S_2 \times S_{n-1}$ (this is realized by letting $L$ be the stabilizer of $\{1,2 \}$ and letting $R$ be the stabilizer of $\{1,2,3\}$).  

Based on the equidistant linear code defined in Example \ref{ex:code}, we define $N_{\alpha}:= \langle (\ell,\ell,\ell,\ell) \mid \ell \in L \rangle$.  Moreover, if $R = R_1 \times R_2,$ where $R_1 \cong S_3$, $R_2 \cong S_{n-1}$, and $R_1 = \langle h, \sigma | h^3 = \sigma^2 = hh^{\sigma} = 1 \rangle,$ we define $N_{\beta}:= \langle (h,h,h, 1), (h,h^{-1},1,h), (x,x,x,x) | x \in \langle \sigma \rangle \times R_2 \rangle.$  By choosing $\sigma\in L$ we have $N_{\alpha} \cap N_{\beta} \cong S_2 \times S_{n-1}$, and, as in \cite[Lemma 5.2]{prodaction}, $N_{\beta} \cong (C_3^2{:}C_2) \times S_{n-1} \not\cong R$. Let $N:= \langle N_{\alpha},N_{\beta} \rangle$. Since $\langle L,R\rangle\cong S_{n+2}$ it follows that $N$ projects onto $S_{n+2}$ in each of its four coordinates. Moreover, given any two of the four coordinates, $N_{\beta}$ contains an element that is the identity in one coordinate and a nonidentity element of $A_{n+2}$ in another. Thus $A_{n+2}^4\norml N$.  Note that $N$ is not necessarily all of $S_{n+2}^4$; indeed, the elements of $N_\beta$ that do not have all entries equal  have even permutations as their entries.

Define $\tau:= (\sigma, 1, \sigma, \sigma)(1,2,3,4).$  Then $\tau^4 = (\sigma, \sigma, \sigma, \sigma)$ and so $\tau^8 = 1.$  Furthermore, $\tau$ centralizes $N_\alpha$ and normalises $N_\beta$.  Let $G_{\alpha}:= \langle N_{\alpha}, \tau \rangle,$ $G_{\beta}:= \langle N_{\beta}, \tau \rangle,$ and $G:= \langle G_{\alpha}, G_{\beta} \rangle.$  By similar reasoning as in \cite[Lemma 5.2]{prodaction}, $A_{n+2}^4 \liso G$  and $G$ induces $C_4$ on the 4 simple direct factors. Moreover, $G_{\beta} \cong \AGL(1,3^2) \times S_{n-1}$.  We also see that $G_{\alpha} \cong C_8 \times S_{n},$ and $G_{\alpha} \cap G_{\beta} \cong C_8 \times S_{n-1}$.  

Let $\Gamma:= \Cos(G,G_{\alpha},G_{\beta})$.  Since $G_{\beta}$ acts on  $[G_{\beta}{:}G_{\alpha} \cap G_{\beta}]$ as $S_{n}$ does on $n$ points and $G_{\alpha}$ acts on $[G_{\alpha}{:}G_{\alpha} \cap G_{\beta}]$ as $\AGL(1,3^2)$ does on $\GF(3^2)$, we see that $\Gamma$ is a connected locally 2-arc-transitive graph with valencies $\{n,9\}.$  Clearly, the action of $G$ on $[G{:}G_{\alpha}]$ is straight diagonal, and the action of $G$ on $[G{:}G_{\beta}]$ is nondiagonal (as in \cite[Lemma 5.2]{prodaction}).  Therefore, $\Gamma$ is a locally 2-arc-transitive graph of straight-nondiagonal type with vertex valencies $\{n,9\}$.
\end{proof}

\subsection{Twisted-nondiagonal type}
As discussed at the start of Section \ref{sec:twtw}, the straight-nondiagonal examples given by Theorem \ref{thm:strnon} can also be viewed as twisted-nondiagonal examples. We also have the following construction of a graph of twisted-nondiagonal type.

\begin{ex}
Let $T = \PSL(2,61).$  By \cite{dickson}, $T$ contains a maximal subgroup $M \cong D_{60}$.  Now, $M$ contains a subgroup $X$ isomorphic to $C_2^2$, and $N_T(X) \cong A_4$.  Now, $N_T(X)$ contains an element $g$ of order three that is not in $M$.  Thus we may select subgroups $L \leqslant M$ and $R \leqslant M^g$ such that $L \cong C_{10} \cong C_5 \times C_2,$ $R \cong C_3{:}C_2,$ $\langle L,R \rangle = T$ and $L \cap R = X \cong C_2.$  Note that we may select presentations $L = \langle l,x | l^5 = x^2 = 1\rangle$ and $R = \langle r,x | r^3 = x^2 = rr^x = 1 \rangle$. 

Note that $L$ has an isomorphism $\phi$ defined by $\phi: l \mapsto l^2, x \mapsto x$.  We define $\overline{l}:= (l, l^{\phi}, l^{\phi^2}, l^{\phi^3}) = (l, l^2, l^4, l^3)$ and $\overline{x} := (x,x,x,x)$.  Furthermore, we define $N_{\alpha}:= \langle \overline{l}, \overline{x}\rangle$, $N_{\beta}:= \langle (r,r,r,1), (r,r^{-1},1,r), \overline{x} \rangle,$ and $N:= \langle N_{\alpha}, N_{\beta} \rangle.$  As in \cite[Lemma 5.2]{prodaction}, none of the coordinates of $N_{\beta}$ can be linked, so $N \cong T^4.$  Moreover, $N_{\alpha} \cong L \cong C_5 \times C_2,$ $N_{\beta} \cong C_3^2{:}C_2$ and $N_{\alpha} \cap N_{\beta} \cong C_2$.

Define $\tau:= (x, 1, x, x)(1,2,3,4).$  Then $\tau^4 = (x,x,x,x)$ and so $\tau^8 = 1.$  Let $G_{\alpha}:= \langle N_{\alpha}, \tau \rangle,$ $G_{\beta}:= \langle N_{\beta}, \tau \rangle,$ and $G:= \langle G_{\alpha}, G_{\beta} \rangle.$  We note that $\tau$ centralizes $\overline{x}$, whereas $\overline{l}^{\tau} = (l^3, l, l^2, l^4) =\overline{l^{3}}$, and so $G_{\alpha} \cong \AGL(1,5).2$.  By similar reasoning as in \cite[Lemma 5.2]{prodaction}, we deduce that $G \cong \PSL(2,61) \Wr C_4$ and $ G_{\beta} \cong \AGL(1,3^2)$.  We also see that $G_{\alpha} \cap G_{\beta} \cong C_8$. 

Let $\Gamma:= \Cos(G,G_{\alpha},G_{\beta})$. Since $G_{\alpha}$ acts on $[G_{\alpha}{:}G_{\alpha} \cap G_{\beta}]$ as $\AGL(1,5)$ does on $\GF(5)$ and $G_{\beta}$ acts on  $[G_{\beta}{:}G_{\alpha} \cap G_{\beta}]$ as $\AGL(1,3^2)$ does on $\GF(3^2)$, we see that $\Gamma$ is a connected locally 2-arc-transitive graph with vertex valencies $\{5,9\}.$  Clearly, the action of $G$ on $[G{:}G_{\alpha}]$ is twisted diagonal, and the action of $G$ on $[G{:}G_{\beta}]$ is nondiagonal (as in \cite[Lemma 5.2]{prodaction}).  Therefore, $\Gamma$ is a locally 2-arc-transitive graph of twisted-nondiagonal type with valencies $\{5,9\}$.  
\end{ex}

\subsection{Nondiagonal-nondiagonal type}

Finally, in this subsection, we include a construction of a graph of nondiagonal-nondiagonal type.

\begin{ex}
Let $T = J_2$, the second Janko group.  By \cite{atlas}, $T$ has two conjugacy classes of elements of order three, labelled 3A and 3B, and two conjugacy classes of involutions, labelled 2A and 2B.  Moreover, the elements of type 3A are contained in a maximal subgroup isomorphic to $A_5 \times D_{10}$ which contains involutions from class 2B, and the elements of type 3B are contained in a maximal subgroup isomorphic to $A_5$ which also contains involutions of type 2B.  Furthermore, within each of these maximal subgroups the elements of order three are normalized by an involution of type 2B.  Using $\GAP$, there are subgroups $L,R < T$, each isomorphic to $S_3$, such that $L \cap R \cong C_2$, $L$ contains an element of order three of type 3A, $R$ contains an element of order three of type 3B, and $\langle L,R \rangle = T.$  Furthermore, by \cite{atlas}, the two conjugacy classes of order three are not fused by any outer automorphism of $T$.  Let $L = \langle l,x | l^3 = x^2 = ll^x = 1 \rangle$ and $R = \langle r,x | r^3 = x^2 = rr^x = 1\rangle$. 

We again use the equidistant linear code as defined in Example \ref{ex:code}.  Define $N_{\alpha}:= \langle (l,l,l,1), (l,l^{-1}, 1, l), (x,x,x,x) \rangle$ and $N_{\beta}:= \langle (r,r,r, 1), (r,r^{-1},1,r), (x,x,x,x) \rangle.$  Note that $L \cap R \cong C_2$, and, reasoning as in \cite[Lemma 5.2]{prodaction}, we deduce that $ N_{\alpha} \cong N_{\beta} \cong C_3^2{:}C_2 \not\cong L,R$. Also, given any two of the four coordinates, both $N_{\alpha}$ and $N_{\beta}$ contain an element that is the identity in one coordinate and a nonidentity element in another, so $N:= \langle N_{\alpha},N_{\beta} \rangle \cong J_2^4$.   

Define $\tau:= (x, 1, x, x)(1,2,3,4).$  Then $\tau^4 = (x,x,x,x)$ and so $\tau^8 = 1.$  Let $G_{\alpha}:= \langle N_{\alpha}, \tau \rangle,$ $G_{\beta}:= \langle N_{\beta}, \tau \rangle,$ and $G:= \langle G_{\alpha}, G_{\beta} \rangle.$  By similar reasoning as in \cite[Lemma 5.2]{prodaction}, $G \cong J_2 \Wr 4$ and $G_{\alpha} \cong G_{\beta} \cong \AGL(1,3^2)$.  We also see that $G_{\alpha} \cap G_{\beta} \cong C_8$.  

Let $\Gamma:= \Cos(G,G_{\alpha},G_{\beta})$. Since $G_{\alpha}$ (respectively $G_{\beta}$) acts on $[G_{\alpha}{:}G_{\alpha} \cap G_{\beta}]$ (respectively $[G_{\beta}{:} G_{\alpha} \cap G_{\beta}]$) as $\AGL(1,3^2)$ does on $\GF(3^2)$, we see that $\Gamma$ is a connected locally $(G,2)$-arc-transitive graph with  valencies $\{9,9\}.$  However, $\Gamma$ cannot be a standard double cover of a $(G,2)$-arc-transitive graph since $L$ and $R$ are not conjugate subgroups in  $\Aut(J_2)$.  Clearly, the action of $G$ on both $[G{:}G_{\alpha}]$ and $[G{:}G_{\beta}]$ is nondiagonal (as in \cite[Lemma 5.2]{prodaction}).  Therefore, $\Gamma$ is a locally $(G,2)$-arc-transitive graph of nondiagonal-nondiagonal type that is regular of valency 9.
\end{ex}

\subsection*{Acknowledgements}

This paper formed part of the Australian Research Council's Discovery Project DP120100446 of the first author.  The authors would also like to thank \'Akos Seress, who made this collaboration possible in 2012 by allowing the second author to visit Australia, and Luke Morgan, for providing a proof of Lemma \ref{lem:Luke} and encouraging the completion of the project.

\bibliographystyle{plain}
\bibliography{PAbib}

\begin{thebibliography}{10}

\bibitem{2arctw}
R.W. Baddeley.
\newblock Two-arc transitive graphs and twisted wreath products.
\newblock {\em J. Algebraic Combin.}, 2:215--237, 1993.

\bibitem{Biggs}
Norman Biggs.
\newblock {\em Algebraic graph theory}.
\newblock Cambridge Mathematical Library. Cambridge University Press,
  Cambridge, second edition, 1993.

\bibitem{BDSL}
Francis Buekenhout, Julie De~Saedeleer, and Dimitri Leemans.
\newblock On the rank two geometries of the groups {${\rm PSL}(2,q)$}: part
  {II}.
\newblock {\em Ars Math. Contemp.}, 6(2):365--388, 2013.

\bibitem{primederange}
T.C. Burness, M.~Giudici, and R.A. Wilson.
\newblock Prime order derangements in primitive permutation groups.
\newblock {\em J. Algebra}, 341:158--178, 2011.

\bibitem{atlas}
J.H. Conway, R.T. Curtis, S.P. Norton, R.A. Parker, and R.A. Wilson.
\newblock {\em Atlas of Finite Groups}.
\newblock Oxford University Press, 2005.

\bibitem{dickson}
Leonard~Eugene Dickson.
\newblock {\em Linear groups: {W}ith an exposition of the {G}alois field
  theory}.
\newblock with an introduction by W. Magnus. Dover Publications Inc., New York,
  1958.

\bibitem{localrees}
X.G. Fang, C.~H. Li, and C.E. Praeger.
\newblock The locally 2-arc transitive graphs admitting a {R}ee simple group.
\newblock {\em J. Algebra}, 282:638--666, 2004.

\bibitem{2arcsuz}
X.G. Fang and C.E. Praeger.
\newblock Finite two-arc transitive graphs admitting a {S}uzuki simple group.
\newblock {\em Comm. Algebra}, 27(8):3727--3754, 1999.

\bibitem{2arcree}
Xin~Gui Fang and Cheryl~E. Praeger.
\newblock Finite two-arc transitive graphs admitting a {R}ee simple group.
\newblock {\em Comm. Algebra}, 27(8):3755--3769, 1999.

\bibitem{GAP4}
The GAP~Group.
\newblock {\em {GAP -- Groups, Algorithms, and Programming, Version 4.4.12}},
  2008.

\bibitem{localsarc}
M.~Giudici, C.H. Li, and C.E. Praeger.
\newblock Analysing finite locally $s$-arc transitive graphs.
\newblock {\em Trans. Amer. Math. Soc.}, 356(1):291--317, 2004.

\bibitem{localstar}
M.~Giudici, C.H. Li, and C.E. Praeger.
\newblock Characterizing finite locally $s$-arc transitive graphs with a star
  normal quotient.
\newblock {\em J. Group Theory}, 9:641--658, 2006.

\bibitem{localdifferent}
M.~Giudici, C.H. Li, and C.E. Praeger.
\newblock Locally $s$-arc transitive graphs with two different quasiprimitive
  actions.
\newblock {\em J. Algebra}, 299:863--890, 2006.

\bibitem{5arc}
Michael Giudici, Cai~Heng Li, and Cheryl~E. Praeger.
\newblock A new family of locally 5-arc transitive graphs.
\newblock {\em European J. Combin.}, 28(2):533--548, 2007.

\bibitem{trivalent}
David~M. Goldschmidt.
\newblock Automorphisms of trivalent graphs.
\newblock {\em Ann. of Math. (2)}, 111(2):377--406, 1980.

\bibitem{2arcpsl}
Akbar Hassani, Luz~R. Nochefranca, and Cheryl~E. Praeger.
\newblock Two-arc transitive graphs admitting a two-dimensional projective
  linear group.
\newblock {\em J. Group Theory}, 2(4):335--353, 1999.

\bibitem{biprimcubic}
A.A. Ivanov and M.E. Iofinova.
\newblock Bi-primitive cubic graphs.
\newblock In I.A. Farad\u{z}ev, A.A. Ivanov, M.H. Klin, and A.J. Woldar,
  editors, {\em Investigations in algebraic theory of combinatorial objects},
  pages 459--472. Kluwer Academic Publishers, 1993.

\bibitem{affine2trans}
A.A. Ivanov and C.E. Praeger.
\newblock On finite affine 2-arc transitive graphs.
\newblock {\em Europ. J. Combin.}, 14:421--444, 1993.

\bibitem{KM}
E.~Kaja and L.~Morgan.
\newblock A new infinite family of star normal quotient graphs of twisted
  wreath type.
\newblock Preprint.

\bibitem{localsporadic}
D.~Leemans.
\newblock Locally $s$-arc transitive graphs related to sporadic simple groups.
\newblock {\em J. Algebra}, 322:882--892, 2009.

\bibitem{LiLingWu}
Cai~Heng Li, Bo~Ling, and Ci~Xuan Wu.
\newblock A new family of quasiprimitive 2-arc-transitive graphs of product
  action type.
\newblock {\em J. Combin. Theory Ser. A}, 150:152--161, 2017.

\bibitem{Li2015}
Cai~Heng Li, \'Akos Seress, and Shu~Jiao Song.
\newblock $s$-arc-transitive graphs and normal subgroups.
\newblock {\em Journal of Algebra}, 421:331--348, Jan 2015.

\bibitem{prodaction}
C.H. Li and \'{A}. Seress.
\newblock Constructions of quasiprimitive two-arc transitive graphs of product
  action type.
\newblock In A.~Hulpke, R.~Liebler, T.~Penttila, and \'{A}. Seress, editors,
  {\em Finite geometries, groups, and computation}, pages 115--123. De
  {G}ruyter, 2006.

\bibitem{LukeLemma}
Luke Morgan.
\newblock Private Communication, March 2022.

\bibitem{quasiprim1}
C.~E. Praeger.
\newblock An {O}'{N}an-{S}cott theorem for finite quasiprimitive permutation
  groups and an application to 2-arc transitive graphs.
\newblock {\em J. London Math. Soc.}, 47:227--239, 1993.

\bibitem{quasiprim2}
C.E. Praeger.
\newblock Finite quasiprimitive graphs.
\newblock In {\em Surveys in combinatorics, 1997. Proceedings of the 16th
  British combinatorial conference}, volume 241 of {\em London Math. Soc.
  Lecture Note Ser.}, pages 65--85. Cambridge University Press, July 1997.

\bibitem{PS}
C.E Praeger and C.~Schneider.
\newblock Embedding permutation groups into wreath products in product action.
\newblock {\em J. Aust. Math. Soc.}, 92:127--136, 2012.

\bibitem{bipartite2arc}
Cheryl~E. Praeger.
\newblock On a reduction theorem for finite, bipartite {$2$}-arc-transitive
  graphs.
\newblock {\em Australas. J. Combin.}, 7:21--36, 1993.

\bibitem{biquasi}
Cheryl~E. Praeger.
\newblock Finite transitive permutation groups and bipartite vertex-transitive
  graphs.
\newblock {\em Illinois J. Math.}, 47(1-2):461--475, 2003.
\newblock Special issue in honor of Reinhold Baer (1902--1979).

\bibitem{scott}
Leonard~L. Scott.
\newblock Representations in characteristic {$p$}.
\newblock In {\em The {S}anta {C}ruz {C}onference on {F}inite {G}roups ({U}niv.
  {C}alifornia, {S}anta {C}ruz, {C}alif., 1979)}, volume~37 of {\em Proc.
  Sympos. Pure Math.}, pages 319--331. Amer. Math. Soc., Providence, R.I.,
  1980.

\bibitem{localsuz}
Eric Swartz.
\newblock The locally 2-arc transitive graphs admitting an almost simple group
  of {S}uzuki type.
\newblock {\em J. Combin. Theory Ser. A}, 119(5):949--976, 2012.

\bibitem{tutte}
W.T. Tutte.
\newblock A family of cubical graphs.
\newblock {\em Proc. Cambridge Philos. Soc.}, 43:459--472, 1947.

\bibitem{tutte2}
W.T. Tutte.
\newblock On the symmetry of cubic graphs.
\newblock {\em Canad. J. Math.}, 11:621--624, 1959.

\bibitem{sleq9}
J.~van Bon and B.~Stellmacher.
\newblock Locally s-transitive graphs.
\newblock {\em Journal of Algebra}, 441:243--293, Nov 2015.

\bibitem{weiss1}
R.~Weiss.
\newblock The nonexistence of 8-transitive graphs.
\newblock {\em Combinatorica}, 1:309--311, 1981.

\bibitem{Wilson}
Robert~A. Wilson.
\newblock Maximal subgroups of sporadic groups.
\newblock In {\em Finite simple groups: thirty years of the atlas and beyond},
  volume 694 of {\em Contemp. Math.}, pages 57--72. Amer. Math. Soc.,
  Providence, RI, 2017.

\end{thebibliography}

\end{document}